\documentclass[12pt]{amsart}%
\usepackage{amssymb}
\usepackage{amsfonts}
\usepackage{amsmath}
\usepackage{graphicx}%
\providecommand{\U}[1]{\protect\rule{.1in}{.1in}}

\makeatletter
\@namedef{subjclassname@2010}{  \textup{2010} Mathematics Subject Classification}
\makeatother
\newtheorem{theorem}{Theorem}[section]
\theoremstyle{plain}

\newtheorem{corollary}[theorem]{Corollary}

\newtheorem{example}[theorem]{Example}

\newtheorem{fact}[theorem]{Fact}
\newtheorem{lemma}[theorem]{Lemma}

\newtheorem{proposition}[theorem]{Proposition}

\numberwithin{equation}{section}

\theoremstyle{definition}
\newtheorem{convention}[theorem]{Convention}
\newtheorem{definition}[theorem]{Definition}
\newtheorem{remark}[theorem]{Remark}

\newcommand {\n}{\mathbb{N}}
\newcommand {\starn}{\,^*\mathbb{N}}

\newcommand {\z}{\mathbb{Z}}
\newcommand {\ld}{\underline{d}}
\newcommand {\ud}{\overline{d}}

\newcommand {\uld}{\overline{ld}}
\newcommand {\lld}{\underline{ld}}

\newcommand {\st}{\operatorname{st}}

\def \BD{\operatorname{BD}}
\def \lBD{\operatorname{\ell BD}}
\def \r{\mathbb{R}}

\begin{document}
\baselineskip=18pt

\title[A monad measure space for logarithmic density]{A monad measure space for logarithmic density}
\author[Di Nasso et. al.]{Mauro Di Nasso, Isaac Goldbring, Renling Jin,
Steven Leth, Martino Lupini, Karl Mahlburg}
\thanks{The authors were supported in part by the American Institute of Mathematics through its SQuaREs program.  I. Goldbring was partially supported by NSF CAREER grant DMS-1349399.  M. Lupini was supported by the York University Susan Mann Dissertation Scholarship.  K. Mahlburg was supported by NSF Grant DMS-1201435.}
\address{Dipartimento di Matematica, Universita' di Pisa, Largo Bruno
Pontecorvo 5, Pisa 56127, Italy}
\email{dinasso@dm.unipi.it}
\address{Department of Mathematics, Statistics, and Computer Science,
University of Illinois at Chicago, Science and Engineering Offices M/C 249,
851 S. Morgan St., Chicago, IL, 60607-7045}
\email{isaac@math.uic.edu}
\address{Department of Mathematics, College of Charleston, Charleston, SC,
29424}
\email{JinR@cofc.edu}
\address{School of Mathematical Sciences, University of Northern Colorado,
Campus Box 122, 510 20th Street, Greeley, CO 80639}
\email{Steven.Leth@unco.edu}
\address{Department of Mathematics and Statistics\\
N520 Ross, 4700 Keele Street\\
Toronto Ontario M3J 1P3, Canada, and Fields Institute for Research in
Mathematical Sciences\\
222 College Street\\
Toronto ON M5T 3J1, Canada}
\email{mlupini@yorku.ca}
\address{Department of Mathematics, Louisiana State University, 228 Lockett
Hall, Baton Rouge, LA 70803}
\email{mahlburg@math.lsu.edu}
\thanks{}
\date{}
\keywords{nonstandard analysis, log density}
\subjclass[2010]{}
\dedicatory{ }

\begin{abstract}
We provide a framework for proofs of  structural theorems about sets with positive Banach logarithmic density.  For example, we prove that if $A\subseteq \n$ has positive Banach logarithmic density, then $A$ contains an approximate geometric progression of any length.  We also prove that if $A,B\subseteq \n$ have positive Banach logarithmic density, then there are arbitrarily long intervals whose gaps on $A\cdot B$ are multiplicatively bounded, a multiplicative version Jin's sumset theorem.    The main technical tool is the use of a quotient of a Loeb measure space with respect to a multiplicative cut.
\end{abstract}

\maketitle

\section{Introduction\label{Section: Introduction and Preliminaries}}

Szemeredi's theorem states that if $A\subseteq \z$ has positive upper density, then $A$ contains arbitrarily large arithmetic progressions.  The main idea behind Furstenberg's proof of Szemeredi's theorem was to associate to the aforementioned set $A$ a dynamical system $(X,\mu,T)$ and a measurable set $E\subseteq X$ with $\ud(A)=\mu(E)$ satisfying, for any finite $F\subseteq \z$:
$$\ud\left(\bigcap_{i\in F}(A-i)\right)\geq \mu\left(\bigcap_{i\in F}T^{-i}(E)\right).$$  This association, now called the \emph{Furstenberg correspondence principle}, converted the task of proving Szemeredi's theorem into the task of proving a theorem of ergodic theory, now referred to as \emph{Furstenberg's multiple recurrence theorem}.  Furstenberg's correspondence principle holds for any countable amenable semigroup (with densities calculated with respect to particular F{\o}lner sequences) and there are many generalizations of Furstenberg's recurrence theorem.  In short, Furstenberg's correspondence has led to a large collection of structural results in combinatorial number theory.

Nonstandard analysis provides an elegant way of establishing Furstenberg's original correspondence theorem.  (For an
introduction to nonstandard methods aimed specifically toward applications to
combinatorial number theory see \cite{jinintro}.)
Indeed, one can consider the \emph{hyperfinite interval} $[-N,N]\subseteq {}^*\z$, equipped with its \emph{Loeb measure} $\mu_L$, which is the $\sigma$-additive measure obtained from the finitely-additive counting measure $\mu(A):=\st(\frac{|A|}{2N+1})$ defined on the algebra of hyperfinite subsets of $[-N,N]$ using the Caratheodory extension theorem.  By the nonstandard characterization of upper density, there is an infinite $N\in {}^*\n$ for which $\ud(A)=\mu_L({}^*A\cap [-N,N])$.  Letting $T:[-N,N-1]\to [-N,N]$ be addition by $1$ (which is easily seen to be measure preserving and defined on a measure $1$ set), the dynamical system $([-N,N],\mu_L,T)$ and the measurable set $E:={}^*A\cap [-N,N]$ witness the conclusion of the Furstenberg correspondence principle.

In this paper, we consider a different kind of density, namely \emph{logarithmic density} (see Section 2 for the precise definition) and seek to associate an appropriate measure space to sets of positive logarithmic density.  Using the nonstandard characterization of logarithmic density, this is accomplished in the same manner as in the previous paragraph.  However, this Loeb measure space contains a serious deficiency, namely the fact that multiplication is not measure preserving.  The main result in this paper is that multiplication is measure-preserving on an appropriate quotient of the associated Loeb measure space.

Initially, we had hoped to use this fact to deduce approximate geometric structure in sets of positive logarithmic density.  Indeed, one can use Furstenberg's multiple recurrence theorem on the quotient space to obtain actual geometric structure in the quotient space, which, when pulled back to the original Loeb space and combined with the transfer principle, would yield approximate geometric structure in the original subset of the integers.  While this process is valid and briefly explained in Section 3, in an upcoming paper we show that we can actually use the original Szemeredi theorem, combined with a ``logarithmic change of coordinates,'' to more directly obtain the aforementioned approximate geometric structure and with better bounds on the nature of the approximation.  Thus, we leave it as an open problem to find more sophisticated applications of the fact that multiplication on our quotient measure space is measure-preserving.

We then briefly discuss a family of densities on subsets of ${\mathbb N}$ for which the corresponding sets of positive measure in the quotient space contain arbitrarily long powers of arithmetic progressions.

In the next to last section, we show that the Lebesgue density theorem is valid in the aforementioned quotient measure space.  In the last section, we use the Lebesgue density theorem to prove a multiplicative analog of a result of Jin \cite{jin}, namely that if $A$ and $B$ both have positive Banach log density, then there are arbitrarily long intervals on which $A\cdot B$ has multiplicatively bounded gaps.

\subsection{Acknowledgements}

This work was initiated during a week-long meeting at the American
Institute for Mathematics on August 4-8, 2014 as part of the SQuaRE (Structured
Quartet Research Ensemble) project \textquotedblleft Nonstandard Methods in
Number Theory.\textquotedblright\ The authors would like to thank the
Institute for the opportunity and for the Institute's hospitality during their stay.

\section{Densities, cuts, and measures
\label{Section: Upper Log Density and Upper Banach Log Density}}

\subsection{Densities}

\begin{convention}
In this paper, $\n$ denotes the set of \emph{positive} natural numbers.
\end{convention}

\noindent For the convenience of the reader, we recall the following:

\begin{definition}
Suppose that $A\subseteq \n$.  Then:
\begin{itemize}
\item The \emph{upper density of $A$} is defined to be
$$\ud(A):=\limsup_{n\rightarrow\infty}\frac{|A\cap [1,n]|}{n}.$$
\item The \emph{lower density of $A$} is defined to be
$$\ld(A):=\liminf_{n\rightarrow\infty}\frac{|A\cap [1,n]|}{n}.$$
\end{itemize}
\end{definition}

\noindent We also recall the definitions of \emph{logarithmic densities}:

\begin{definition}
Suppose that $A\subseteq \n$.  Then:
\begin{itemize}
\item The \emph{upper logarithmic density of $A$} is defined to be
$$\uld(A):=\limsup_{n\rightarrow\infty}\frac{1}{\ln n}\sum_{x\in A\cap [1,n]}\frac{1}{x}.$$
\item The \emph{lower logarithmic density of $A$} is defined to be
$$\lld(A):=\liminf_{n\rightarrow\infty}\frac{1}{\ln n}\sum_{x\in A\cap [1,n]}\frac{1}{x}.$$
\end{itemize}
\end{definition}

%
%

When dealing with logarithmic densities, it is useful to recall that, setting $H_n:=\sum_{k=1}^n \frac{1}{k}$ (the so-called \emph{$n^{\text{th}}$ harmonic number}), we have $\lim_{n\to \infty}(H_n -\ln n)=\gamma$, the so-called \emph{Euler-Mascheroni} constant.  For example, it follows easily that $\uld(\n)=\lld(\n)=1$.

The proof of the following lemma is straightforward.

\begin{lemma}
Suppose that $A,B\subseteq \n$ and $n\in \n$.
\begin{enumerate}
\item $\uld(A+n)=\uld(A)$ and $\lld(A+n)=\lld(A)$.
\item If $A\triangle B$ is finite, then $\uld(A)=\uld(B)$ and $\lld(A)=\lld(B)$.
\end{enumerate}
\end{lemma}

The following fact is the content of \cite[Lemma 2.1(e)(f)]{BBHS2}:

\begin{fact}\label{comparingdensities}
For $A\subseteq \n$, we have $\ld(A)\leqslant\lld(A)\leqslant\uld(A)\leqslant\ud(A)$.
\end{fact}

We would like to offer an alternative proof of the preceding fact.  We will only prove that $\ld(A)\leq \lld(A)$; the other inequality follows from the inequality for lower densities and the fact that $\ud(A)=1-\ld(\n \setminus A)$ and $\uld(A)=1-\lld(\n\setminus A)$.  The heuristic behind our proof is simple:  the logarithmic density of a set can only decrease if we ``push the elements of the set to the right;'' such a shift should leave the lower density fixed.  Here are the specifics:

Set $f_A:\n\to \r$ to be defined by $f_A(n):=\sum_{x\in A\cap [1,n]}\frac{1}{x}$. Without loss of generality, we may assume $\ld(A)>0$.  Take $\alpha<\ld(A)$ and $H>\n$.  It suffices to show that $\st(\frac{f_A(H)}{\ln H})\geq \alpha$.  Since two sets that differ by only a finite number of elements have the same lower density and the same lower logarithmic density, we can assume that
$\inf_{n\geqslant 1}\frac{|A\cap [1,n]|}{n}>\alpha$.  

Let $m:=|{}^{\ast }\!{A}\cap [1,H]|$ and set $$B=\left\{\left\lfloor \frac{x}{\alpha}\right\rfloor+1:x\in [1,m]\right\}\cap [1,H].$$

\

Next observe that, for every $k\in [1,H]$, we have $\frac{|B\cap [1,k]|}{k}\leq \alpha$.  (Without taking integer parts, $B$ would be an arithmetic progression of real numbers, whence the densities are clearly bounded by $\alpha$; by taking integer parts and then adding $1$, if anything, we have reduced the densities.)
%
%
Let $K:=|B|$.  Let $(a_n \ : \ n\leq m)$ and $(b_n \ : \ n\leq K)$ be the enumerations of $A\cap [1,H]$ and $B$ in increasing order.  Since $\alpha<\frac{|{}^{\ast }\!{A}\cap [1,k]|}{k}$ for each $k\in [1,H]$, it follows that $a_n\leq b_n$ for all $n\leq K$.  We thus get that
$$f_A(H)=\sum_{n=1}^m \frac{1}{a_n}\geq \sum_{n=1}^K \frac{1}{a_n}\geq \sum_{n=1}^K \frac{1}{b_n}=:f_B(H).$$
Since $\frac{f_B(H)}{\ln H}\approx\frac{\alpha\ln(H)}{\ln H}=\alpha$, it follows that $\st(\frac{f_A(H)}{\ln H})\geq \alpha$. 

\


\noindent We also recall the following definition:

\begin{definition}
For $A\subseteq \n$, the \emph{(upper) Banach density} of $A$ is defined to be
$$\BD(A):=\lim_{n\to \infty} \sup_{k\geq 1}\frac{|A\cap [k,k+n]|}{n+1}.$$
\end{definition}

Of course, for the preceding definition to be legitimate, one must prove that the limit involved always exists.  This is a rather straightforward argument; it also follows immediately from Fekete's Lemma (see \cite{Fekete}).

We now want to define a Banach version of logarithmic density; to do so, we must show that the corresponding limit exists.

\begin{lemma}\label{analysislemma}
Suppose that $g:\n\to \r$ is a nondecreasing function satisfying, for all $j,n\in \n$, the inequality $g(n^j)\leq jg(n)$.  Then $\lim_{n\to \infty}\frac{g(n)}{\ln n}$ exists and equals $\inf_{n\geq 1} \frac{g(n)}{\ln n}$.
\end{lemma}

\begin{proof}
It is enough to show that, for every $n\in \n$ and $N\in {}^*\n\setminus \n$, we have $\st(\frac{g(N)}{\ln N})\leq \frac{g(n)}{\ln n}$.  Take $j\in {}^*\n$ such that $n^j\leq N<n^{j+1}$; note that $j>\n$.  We conclude by observing that
$$\frac{g(N)}{\ln N}\leq \frac{(j+1)g(n)}{j\ln n}=(1+\frac{1}{j})\frac{g(n)}{\ln n}\approx \frac{g(n)}{\ln n}.$$
\end{proof}

\begin{proposition}
For any $A\subseteq \n$, the limit
$$\lim_{n\to \infty}\sup_{k\geq 1}\frac{1}{\ln n}\sum_{x\in A\cap [k,nk]}\frac{1}{x}$$ exists and equals $$\inf_{n\geq 1} \sup_{k\geq 1}\frac{1}{\ln n}\sum_{x\in A\cap [k,nk]}\frac{1}{x}.$$
\end{proposition}

\begin{proof}
Define $g:\n\to \r$ by
$$g(n)\ =\ \sup_{k\ge 1}\left(\sum_{x\in[k,kn)\cap A}\frac{1}{x}\right).$$  Clearly $g$ is nondecreasing, so, by Lemma \ref{analysislemma}, it suffices to show that $g(n^j)\leq jg(n)$ for all $j,n\in \n$.
To see this, it suffices to observe that, for a fixed $k$, one has
$$\sum_{x\in[k,kn^j)\cap A}\frac{1}{x}\ =\ 
\sum_{s=1}^{j}\left(\sum_{x\in[kn^{s-1},kn^s)\cap A}\frac{1}{x}\right)\ \le\ 
\sum_{s=1}^j g(n)\ =\ j\cdot g(n).$$
%

\end{proof}

%

\noindent We are thus entitled to make the following:

\begin{definition}
For $A\subseteq \n$, the \emph{(upper) Banach log density} of $A$ is 
$$\lBD(A):=\lim_{n\to \infty}\sup_{k\geq 1}\frac{1}{\ln n}\sum_{x\in A\cap [k,nk]}\frac{1}{x}.$$
\end{definition}

Of course one could also define the lower Banach log density, but in this paper we only focus on the upper Banach log density.

The next proposition can be proven in a manner analogous to the corresponding statement for upper log density.
\begin{proposition}
For any $A\subseteq\n$, we have $\lBD(A)\leq \BD(A)$.
\end{proposition}

Finally, we will frequently make use of the following nonstandard formulation of Banach log density.
\begin{proposition}\label{nsubld}
If $A\subseteq \n$, then $\lBD(A)\geq \alpha$ if and only if for every $N>\n$, there is $k\in {}^*\n$ such that $$\st\left(\frac{\sum_{x\in {}^{\ast }\!{A}\cap [k,Nk]} \frac{1}{x}}{\ln N}\right)\geq \alpha.$$
\end{proposition}

\subsection{Multiplicative cuts}
\begin{definition}
An infinite initial segment $V$ of $\starn$
is a \emph{multiplicative cut} if $V\!\cdot\!V\subseteq V$.
\end{definition}

\noindent Note the following obvious facts:

\begin{itemize}
\item multiplicative cuts are also additive cuts, that is, they are closed under addition;
\item bounded multiplicative cuts must be external;
\item $\n$ is the smallest multiplicative cut. 
\end{itemize}

%

For $N\in {}^*\n\setminus \n$, we let
\begin{equation}\label{V_N}
V_N=\bigcap_{n\in\n}[1,\lfloor N^{1/n}\rfloor ].
\end{equation}
Then $V_N$ is the largest multiplicative cut
in $[1,N]$.

\begin{definition}
Suppose that $U$ and $V$ are infinite initial segments of $\starn\cup\{0\}$ and $\starn$ respectively.  We set:
\begin{enumerate}
\item $\ln V:=\{x\in\starn\cup\{0\}:\lfloor e^x\rfloor\in V\}.$
\item $e^U=\bigcup_{x\in U}[1, \lfloor e^x\rfloor ].$
\end{enumerate}
\end{definition}

\noindent It is straightforward to verify the following facts:
\begin{enumerate}
\item $V$ is a multiplicative cut if and only if $\ln V$ is an additive cut.
\item $e^U$ is a multiplicative cut if and only if $U$ is an additive cut.
\item If $U$ is an additive cut, then $\ln(e^U)=U$. 
\item If $V$ is a multiplicative cut, then $e^{\ln V}=V$.
\end{enumerate}

In the rest of this subsection, we fix $N\in \starn\setminus \n$ and a multipicative cut $V\subseteq [1,N]$.

\begin{definition}
For any $a,b\in\starn\setminus \n$, we declare $a\sim_V b$ if and only if $|\lfloor \ln a\rfloor-\lfloor \ln b\rfloor|\in \ln V$.
\end{definition}

Equivalently, if $a<b$, then $a\sim_V b$ if and only if $\lfloor \frac{b}{a}\rfloor \in V$.
Note that $\sim_V$ is an equivalence relation on $\starn$.  For $a\in \starn$, we set $[a]^V:=\{x\in\starn :a\sim_V x\}$.  We also set $\varphi_V:{}^*\n\to {}^*\n/\sim_V$ to denote the quotient map, that is, $\varphi_V(a):=[a]^V$.  If $V=\n$, we simply write $\varphi$ instead of $\varphi_V$.

The proof of the following proposition is straightforward.

\begin{proposition}
Fix $a\in \starn$. Then:
\begin{enumerate}
\item If $x,y\in [a]^V$ and $x<y$, then $[x,y]\subseteq [a]^V$.
\item $[a]^V=\bigcup_{x\in V}
\left[\lfloor ax^{-1}\rfloor,ax\right]$.
\end{enumerate}
\end{proposition}

It is straightforward to show that, if $[a]^V=[a']^V$ and $[b]^V=[b']^V$, then $[ab]^V=[a'b']^V$.  (For instance, use that equality modulo an additive cut is a congruence relation with respect to addition on $\starn$.)  This allows us to set, for $a,b\in \starn$,  $[a]^V\cdot [b]^V:=[ab]^V$.  It is worth noting that this multiplication on equivalence classes satisfies cancellation:  if $[a]^V\cdot [b]^V=[a]^V\cdot [c]^V$, then $[b]^V=[c]^V$.

We can also order equivalence classes by setting $[a]^V<[b]^V$ if and only if $a<b$ and $a\not\sim_V b$.

\begin{proposition}
$(\starn/\sim_V,<)$ is a dense linear order.
\end{proposition}

\begin{proof}
Suppose that $[a]^V<[b]^V$.  Let $c:=\lfloor \sqrt{ab}\rfloor$.  It is readily verified (using that $V$ is a multiplicative cut) that $[a]^V<[c]^V<[b]^V$.
\end{proof}

For any $k\in\starn$, we set $\mathcal{H}_{k,N,V}:=\varphi_V([k,Nk])$.  Once again, to simplify notation, if $V=\n$, we simply drop the $V$ and write $\mathcal{H}_{k,N}$ instead of $\mathcal{H}_{k,N,V}$.  We will often abuse notation and write $\varphi_V:[k,Nk]\to \mathcal{H}_{k,N,V}$, that is, we will let $\varphi_V$ also denote its restriction to $[k,Nk]$.

It is worth noting that if $a\in [k,Nk]$ is such that $ax>Nk$ or $\lfloor \frac{a}{x}\rfloor<k$ for some $x\in V$, then $[a]^V$ is not completely contained in $[k,Nk]$; for our purposes, the set of such exceptional $a$'s will become negligible in a sense to be made precise shortly.  In light of Proposition \ref{nsubld}, the spaces $\mathcal{H}_{k,N,V}$ will prove important when studying Banach log density.

\begin{remark}
For each $a\in [k,Nk]$, set
\begin{equation}\label{starndardize}
\Phi(a):=\st\left(\frac{\ln a-\ln k}{\ln N}\right).
\end{equation}
Then $\Phi:[k,Nk]\to [0,1]$ is easily seen to be a surjection. Moreover, $\Phi(a)=\Phi(b)$ if and only if $a\sim_{V_N}b$, 
where $V_N$ is defined as in (\ref{V_N}).
Hence, we obtain an order-preserving isomorphism $\Phi_{\#}:\mathcal{H}_{k,N,V_N}\to[0,1]$ given by \[\Phi_{\#}([a]^{V_N}):=\st\left(\frac{\ln a-\ln k}{\ln N}\right).\]
\end{remark}

\subsection{Loeb measure spaces}

For each internal set $A\subseteq [k,Nk]$, set
\[\nu(A):=\nu_{k,N}(A)=\st\left(\sum_{a\in A}\frac{1}{a\ln N}\right).\]
It is readily verified that $\nu$ is a finitely additive measure defined on the internal subsets of $[k,Nk]$, whence we obtain a Loeb measure space based on $[k,Nk]$, whose measure we continue to denote by $\nu=\nu_{k,N}$.  By Proposition \ref{nsubld}, for every $N>\n$, there is $k\in \starn$ such that $\lBD(A)=\nu_{k,N}({}^{\ast }\!{A}\cap [k,Nk])$.

Recall that, for $n\in \n$, we set $H_n=\sum_{k=1}^n \frac{1}{k}$.
\begin{proposition}
For any $k\leqslant a\leqslant b\leqslant Nk$, we have $$\nu([a,b])=\st\left(\frac{\ln b-\ln a}{\ln N}\right).$$  In particular, $\nu([k, k\sqrt{N}])=\nu([ k\sqrt{N},Nk])=\frac{1}{2}$.
\end{proposition}

\begin{proof}
We assume that $a,b\in \starn\setminus \n$; the other cases are similar and easier.  We have
$$\frac{H_b-H_{a-1}}{\ln N}=\frac{(H_b-\ln b)+(\ln b-\ln a)+(\ln \frac{a}{a-1})+(\ln (a-1)-H_{a-1})}{\ln N}.$$  Since $a,b>\n$, we have $H_{a-1}- \ln(a-1),H_b- \ln b\approx \gamma$.  Also, $\ln \frac{a}{a-1}\approx 0$.  It follows that 
$$\nu([a,b])=\st\left(\sum_{x=a}^b \frac{1}{x\ln N}\right)=\st\left(\frac{H_b-H_{a-1}}{\ln N}\right)=\st\left(\frac{\ln b-\ln a}{\ln N}\right).$$
\end{proof}

\medskip

\begin{corollary}
Suppose that $a,b,c\in{}^*\n$ are such that $a,b,ac,bc\in [k,Nk]$.  Then $\nu([a,b])=\nu([ac,bc])$.
\end{corollary}

In contrast to the previous corollary, note that, under the same assumptions, $\nu(c\cdot [a,b])\not=\nu([a,b])$ in general, that is, multiplication need not be measure preserving.  Indeed,
$$\nu(c\cdot [a,b])=\st\left(\sum_{x\in [a,b]}\frac{1}{cx\ln N}\right)=\st\left(\frac{1}{c}\sum_{x\in [a,b]}\frac{1}{x\ln N}\right).$$  We will shortly see that this problem vanishes when we pass to the quotient space $\mathcal{H}_{k,N,V}$.

In calculations pertaining to the quotient space $\mathcal{H}_{k,N,V}$, it will become useful to know how to approximate the measures of certain internal subsets of $[k,Nk]$.  First, let us establish some notation.  We call an interval $[a,b]\subseteq [k,Nk]$ \emph{big} if $\st(\frac{b}{a})>2$ (where, for the sake of this definition, the standard part of an infinite hyperreal is itself).  Now suppose that $C\subseteq [k,Nk]$ is internal and we write $C=\bigsqcup_{i\in I} [a_i,b_i]$, where the intervals $[a_i,b_i]$  are the internal connected components of $C$, that is, they are the maximal intervals contained in $C$.  (Note then that the set $I$ and the sequences $(a_i)$ and $(b_i)$ are all internal.)  We then say that $C$ \emph{has big components} if each connected component $[a_i,b_i]$ is big. 

For the proof of the next lemma, we will need to recall the following elementary estimates:  suppose that $r,s\in \n$ are such that $2\leq r\leq s$.  Then:
$$\ln(s+1)-\ln(r)\leq \sum_{i=r}^s \frac{1}{i}\leq \ln(s)-\ln(r-1).$$

\begin{lemma}\label{bigestimate}
Suppose that $C=\bigsqcup_{i\in I} [a_i,b_i]$ has big components and that $C\subseteq {}^*\n\setminus \n$.  Then $\nu(C)\approx \frac{1}{\ln N}\sum_{i\in I}(\ln(b_i)-\ln(a_i))$. 
\end{lemma}

\begin{proof}
Fix $i\in I$.  Note then that $$\ln(b_i)-\ln(a_i)\leq \sum_{n\in [a_i,b_i]} \frac{1}{n}\leq \ln(b_i)-\ln(a_i-1).$$  It follows that $$\frac{|(\ln(b_i)-\ln(a_i))-\sum_{n\in [a_i,b_i]}\frac{1}{n}|}{\ln(b_i)-\ln(a_i)}\leq \frac{\ln(a_i)-\ln(a_i-1)}{\ln 2}\approx 0.$$  Fix $\epsilon>0$.  We then have
\begin{alignat}{2}
|(\sum_{i\in I} (\ln(b_i)-\ln(a_i)))-(\sum_{i\in I}\sum_{n\in [a_i,b_i]}\frac{1}{n})|&\leq \sum_{i\in I}\epsilon \cdot (\ln(b_i)-\ln(a_i))\notag \\ \notag
								&\leq \epsilon \cdot \sum_{i\in I}\sum_{n=a_i+1}^{b_i} \frac{1}{n}\\ \notag
								&\leq \epsilon \cdot H_N.\notag
\end{alignat}

Therefore, we have
$$|\nu(C)- \frac{1}{\ln N}\sum_{i\in I}(\ln(b_i)-\ln(a_i))|\leq 2\epsilon \cdot \frac{H_N}{\ln N}\approx 2\epsilon.$$  Since $\epsilon>0$ was arbitrary, this yields the desired result.
\end{proof}

\subsection{Quotient measure spaces}

Let $V$ be a multiplicative cut contained in $[1,N]$. 
Via $\varphi:[k,Nk]\to\mathcal{H}_{k,N,V}$, the Loeb measure $\nu_{k,N}$ induces a measure $\frak{m}=\frak{m}_{k,N,V}$
on $\mathcal{H}_{k,N,V}$.  More precisely, a set $E\subseteq\mathcal{H}_{k,N,V}$ is $\mathfrak{m}_{k,N,V}$-measurable if and only if
$\varphi^{-1}(E)$ is $\nu_{k,N}$-measurable, in which case we set
\begin{equation}
\mathfrak{m}_{k,N,V}(E):=\nu_{k,N}(\varphi^{-1}(E)).
\end{equation}
Of course, $\mathfrak{m}_{k,N,V}$ is a probability measure on $\mathcal{H}_{k,N,V}$.  Since Loeb measures are complete, it follows that $\frak{m}_{k,N,V}$ is also complete.  As before, if $V=\n$, then we write $\mathfrak{m}_{k,N}$ instead of $\mathfrak{m}_{k,N,\n}$.

\begin{example}
If $V=V_N$, then the order-preserving isomorphism $\Phi_{\#}:\mathcal{H}_{k,N,V_N}\to [0,1]$ is also an isomorphism of measure spaces, where $[0,1]$ is equipped with the usual Lebesgue measure.
\end{example}

\begin{proposition}\label{internalmeasurable}
Suppose that $A\subseteq [k,Nk]$ is internal.  Then $\varphi(A)$ is $\frak{m}$-measurable.
\end{proposition}

\begin{proof}
The proof is identical to that of \cite[Proposition 6.3]{DGJLLM}.
\end{proof}



Recall that if $(X,\mathcal B,\mu)$ and $(Y,\mathcal C,\nu)$ are probability spaces, then $T:X\to Y$ is said to be \emph{measure-preserving} if $T$ is measurable and $\mu(T^{-1}(A))=\nu(A)$ for all $A\in \mathcal C$.  If, additionally, $T^{-1}$ exists $\nu$-almost everywhere and is also measure-preserving, then we say that $T$ is an \emph{invertible measure-preserving} map.

Given $x:=[a]^V$, we can define a map $T_x:\mathcal{H}_{k,N,V}\to \mathcal{H}_{ka,N,V}$ by $T_x(e):=xe$.

\begin{proposition}\label{measpresgen}
For any $x:=[a]^V$, we have $T_x:\mathcal{H}_{k,N,V}\to \mathcal{H}_{ka,N,V}$ is an invertible measure-preserving map.
\end{proposition}

\begin{proof}
We will only show:  if $E\subseteq \mathcal{H}_{k,N,V}$ is $\frak{m}_{k,N,V}$-measurable, then $T_x(E)$ is $\frak{m}_{ka,N,V}$-measurable and $\frak{m}_{ka,N,V}(T_x(E))=\frak{m}_{k,N,V}(E)$.  To finish the proof of the proposition, one would need to show that $T_x$ is measurable and measure-preserving; the proof of this fact is similar to what we will actually show but is a bit messier.  

Without loss of generality, we may suppose that $a\in {}^*\n\setminus \n$.  Indeed, if $a\in \n$, then $T_x$ is ``essentially'' the identity map on $\mathcal{H}_{k,N,V}$; see the discussion following the proof of the current proposition.

Without loss of generality, we may also assume that $\varphi_V^{-1}(E)\subseteq {}^*\n\setminus \n$.  Fix (standard) $\epsilon>0$.  Since $\varphi_V^{-1}(E)$ is Loeb measurable, we can find internal sets $C,D\subseteq [k,Nk]$ with $C\subseteq \varphi_V^{-1}(E)\subseteq D$ and with $\nu_{k,N}(D\setminus C)<\epsilon$.  Without loss of generality, we may assume that $D\subseteq {}^*\n\setminus \n$ and that both $C$ and $D$ have big components.  Indeed, we can arrange that $D$ has big components by deleting from $D$ all of the components that are not big; note that the remaining set is internal and still contains $\varphi_V^{-1}(E)$.  We can arrange that $C$ has big components by prolonging each connected component to three times the right endpoint (and merging intervals where necessary); the resulting set is still internal, is still contained in $\varphi_V^{-1}(E)$, and is readily verified to have big components.

Decompose $C=\bigsqcup_{i\in I}[a_i,b_i]$ and $D=\bigsqcup_{j\in J}[c_j,d_j]$  into their connected components.  Set $F:=\bigsqcup_{i\in I}[aa_i,ab_i]$ and $G:=\bigsqcup_{j\in J}[ac_j,ad_j]$.

\

\noindent \textbf{Claim:}  $F\subseteq \varphi_V^{-1}(T_x(E))\subseteq G$.

\noindent \textbf{Proof of Claim:}  First suppose that $p\in F$. Fix $l\in [a_i,b_i]$ such that $al\leq p\leq a(l+1)$.  Since $l\sim_V l+1$, we have $al\sim_V a(l+1)$, whence $[p]^V=[al]^V\in T_x(E)$.  Now suppose that $p\in \varphi_V^{-1}(T_x(E))$, say $[p]^V=[a]^V\cdot [d]^V$ with $[d]^V\in E$.  Take $y\in V$ such that $p\in [\lfloor \frac{ad}{y}\rfloor, ady]$.  Since $a>\n$, we have $a\lfloor \frac{d}{y}\rfloor \leq \lfloor \frac{ad}{y}\rfloor$, whence we have $p\in [a\lfloor \frac{d}{y}\rfloor,ady]$.  Write $p=ak+r$ with $k\in [\lfloor \frac{d}{y}\rfloor,dy]$ and $0\leq r<a$.  Since $[\lfloor \frac{d}{y}\rfloor,dy]\subseteq \varphi_V^{-1}(E)$, we have $k\in [c_j,d_j]$ for some $j\in J$.  Note that $d_j\notin \varphi_V^{-1}(E)$ as then $d_j+1\in \varphi_V^{-1}(E)\subseteq D$, a contradiction.  Thus $p=ak+r\leq a(d_j-1)+a=ad_j$, whence $p\in [ac_j,ad_j]$.  This completes the proof of the claim.

\

\noindent Since $F$ has big components and is contained in ${}^*\n\setminus \n$, by Lemma \ref{bigestimate} we have that 
$$\nu_{ka,N}(F)\approx \frac{1}{\ln N}\sum_{i\in I} (\ln(ab_i)-\ln(aa_i))=\frac{1}{\ln N}(\ln(b_i)-\ln(a_i))\approx \nu_{k,N}(C).$$  We conclude that $\nu_{ka,N}(G)=\nu_{k,N}(D)$.  For the same reason, we have that $\nu_{ka,N}(G)= \nu_{k,N}(D)$.  

%
			   
It follows that $\nu_{ka,N}(G\setminus F)<\epsilon$.  Since $\epsilon>0$ was arbitrary, this shows that $\varphi^{-1}_V(T_x(E))$ is Loeb measurable.  Moreover, $$|\nu_{ka,N}(\varphi^{-1}_V(T_x(E)))-\nu_{k,N}(\varphi^{-1}_V(E))|\leq |\nu_{ka,N}(G)-\nu_{k,N}(D)|+2\epsilon=2\epsilon;$$ since $\epsilon>0$ is arbitary, we have $\nu_{ka,N}(\varphi^{-1}(T_x(E)))=\nu_{k,N}(\varphi^{-1}(E))$, that is, $\frak{m}_{ka,N,V}(T_x(E))=\frak{m}_{k,N,V}(E)$.
\end{proof}

Now suppose that $x:=[a]^V$ is such that $a<V_N$, where $V_N$ is as in Equation (\ref{V_N}).  Then the ``inclusion'' mapping $i:[k,Nk]\to [ka,Nka]$ is defined on a conull set and is an invertible measure-preserving transformation.  In this way, we can identify the measure spaces $\mathcal{H}_{k,N,V}$ and $\mathcal{H}_{ka,N,V}$.  Combining this identification and Proposition \ref{measpresgen}, we obtain the following:

\begin{proposition}\label{measurepreserving}
For $x:=[a]^V$ with $a<V_N$, the map $T_x:\mathcal{H}_{k,N,V}\to \mathcal{H}_{k,N,V}$ is an invertible measure-preserving transformation.
\end{proposition}

\section{Geo-arithmetic progressions}

In this short section, we indicate how our results from the previous section can be used to obtain approximate geometric structure in sets of positive Banach log density.  As mentioned in the introduction, in an upcoming paper we show how stronger results can be deduced from Szemeredi's theorem and a logarithmic change of coordinates.

Let $x,a\in\starn$. If $n\in \n$, we say that $x$ is an \emph{$n$-approximation} of $a$ if $x/n<a<xn$.  If every element $x\in X$ is an $n$-approximation of some $a\in A$, we say that
$X$ is an \emph{$n$-approximate subset} of $A$.

For the convenience of the reader, we recall:

\begin{fact}[Furstenberg's Recurrence Theorem]\label{furstenberg}
Let $T:X\to X$ be a measure-preserving transformation on the probability space $(X,\mathcal{B},\mu)$.  Further suppose that $A\in \mathcal {B}$ satisfies $\mu(A)>0$ and $l\in \n$ is given.  Then there exists $n\in \n$ such that
\[\mu(A\cap T^{-n}(A)\cap T^{-2n}(A)\cap \cdots \cap T^{-ln}(A))>0.\]
\end{fact}

%

\begin{theorem}\label{geometric}
Let $A\subseteq\n$ be such that $\lBD(A)>0$ and fix $l\in\n$. Then there exists $n\in \n$ such that,
for any $m\in \n$, there exists a geometric progression $G=\{ar^i:i=0,1,\ldots,l-1\}$ with
$a,r>m$ such that $G$ is an $n$-approximate subset of $A$.
\end{theorem}

\begin{proof}
Set $\alpha:=\lBD(A)$.  Take $k,N\in \starn$ with $N>\n$ such that $\alpha=\nu_{k,N}({}^{\ast }\!{A}\cap [k,Nk])$.
Let $E=\varphi({}^{\ast }\!{A}\cap [k,Nk])\subseteq\mathcal{H}_{k,N}$. By Proposition \ref{internalmeasurable}, we have that $E$ is $\frak{m}_{k,N}$-measurable and that
$\mathfrak{m}_{k,N}(E)\geqslant\alpha$. Fix $s\in \starn$ with $\n<s<V_N$ and set $x:=[s]^\n$.  By Furstenberg's Recurrence Theorem applied to the transformation $T_x$ on $\mathcal{H}_{k,N}$ (which is applicable by Proposition \ref{measurepreserving}), we see that
$E$ contains a geometric progression
$\{cq^i:i=1,2,\ldots,l\}$; here, $q=x^k$ for some $k\in \n$.  Let $r:=s^k$.
Choose any $a\in\varphi^{-1}(cq)$. Then $a>\n$ and
$\varphi(ar^{i-1})=cq^i$.
Let $n_i=\min\{j\in \n:[\lfloor \frac{ar^i}{j}\rfloor,ar^ij]\cap\,{}^{\ast }\!{A}\not=\emptyset\}$.
Set $n=\max\{n_i:i=0,1,\ldots,l-1\}$. We now conclude that
there exists an $l$-term geometric progression in $[k,Nk]$
with infinite ratio and infinite initial element
such that every term in the progression is an $n$-approximation
of some element in ${}^{\ast }\!{A}\cap [k,Nk]$. The theorem follows by the transfer principle.
\end{proof}

We give two examples to show the necessity of some of the statements in the previous theorem.  First, we show that we can only expect to get approximate arithmetic progressions in general.

\begin{example}
Let $A$ be the set of all square-free numbers. Then by Fact \ref{comparingdensities} we have $\lBD(A)\geqslant\uld(A)\geqslant
\ld(A)>0$ but $A$ does not contain any $3$-term geometric progression.
\end{example}

The next example shows that we really do need the Banach log density to be positive.
\begin{example}
Let $\alpha<1$. Fix a $j$ such that $(j-1)/j>\alpha$. Let $u_0=2$,
$u_{i+1}>(ju_i)^3$, and set
\[A=\bigcup_{i=1}^{\infty}[u_i,ju_i].\]
(Observe that $\ud(A)>\alpha$ but $\lBD(A)=0$.) For any $n\in\n$, there exists an $m\in\n$ such that
there does not exist $3$-term geometric progression $G=\{a,ar,ar^2\}$ with $a,r>m$
 and $G$ is an $n$-approximate subset of $A$.
\end{example}

\begin{proof}
Let $m=n^3j$. Let $a,r>m$ and $G=\{a,ar,ar^2\}$ be a $3$-term geometric progression
such that $u_{i_1}/n\leqslant a\leqslant ju_{i_1}n$ and
$u_{i_2}/n\leqslant ar\leqslant ju_{i_2}n$.

If $i_1=i_2$, then we get $\frac{u_{i_1}}{n}\leq a,ar\leq ju_{i_1}n$, whence $r\leq n^2j=m$, a contradiction. So we can assume that $i_2>i_1$. Then
$\frac{u_{i_2}}{ju_{i_1}n^2}\leqslant r\leqslant\frac{ju_{i_2}n^2}{u_{i_1}}$. Hence, it is readily verified that
\[ju_{i_2}n<u_{i_2}\frac{u_{i_2}}{ju_{i_1}n^3}\leqslant ar^2
\leqslant ju_{i_2}\frac{ju_{i_2}n^3}{u_{i_1}}<u_{i_2+1}\frac{n^3}{ju_{i_2}u_{i_1}}
<u_{i_2+1}/n.\]
Therefore, $G$ is not an $n$-approximate subset of $A$.
\end{proof}

\section{Other densities\label{Section: Other densities}}

In this section, we introduce a family of densities on subsets of $\n$ for which the corresponding sets of positive measure in the quotient space contain arbitrarily long powers of arithmetic progressions.  Since many of the properties of these densities have proofs analogous to the case of logarithmic density, we allow ourselves to just state the main definitions and results and omit almost all proofs.

\begin{definition}
For any positive integer $m$ and any set $A\subseteq\n$ let
\[\BD_m(A):=\lim_{n\to \infty} \sup_{k\in\n}\frac{1}{mn}\sum_{x\in A\cap [k,(\lceil\!\sqrt[m]{k}\rceil +n)^m]}\frac{1}{x^{\frac{m-1}{m}}}.\]
\end{definition}

\noindent Clearly, $\BD_1(A)=\BD(A)$.

%
%

\begin{definition}\label{monadformpower}
Fix $m\in\n$, $N\in\starn\setminus\n$, and $k\in\starn$.
Let $U\subseteq [1,N]$ be an additive cut (for example, $U=\n$). Let
\[I_{k,N,m}:=[k,(\lceil\!\sqrt[m]{k}\rceil+N)^m].\] For any $a,b\in I_{k,N,m}$,
set $a\sim b$ if $|\sqrt[m]{a}-\sqrt[m]{b}|<u$ for some $u\in U$.
Let  \[[a]_m:=\{x\in I_{k,N,m}:x\sim a\}.\]
\end{definition}
Clearly, if $x,y\in [a]$ and $x<y$, then $[x,y]\subseteq [a]$.

\begin{proposition}
The relation $\sim$ is an equivalence relation.
\end{proposition}

The monad $[a]$ is the set $\left(\lceil\!\sqrt[m]{a}\rceil \pm U\right)^m$ where
\[\left(\lceil\!\sqrt[m]{a}\rceil \pm U\right)^m:=\left(\bigcup_{u\in U}
\left[\left(\lceil\!\sqrt[m]{a}\rceil-u\right)^m,\left(\lceil\!\sqrt[m]{a}\rceil+u\right)^m\right]
\right)\cap I_{k,N,m}.\]

\begin{definition}
Let $m,N,k,U$ be the same as in Definition \ref{monadformpower}. Let
\[\mathcal{G}_{k,N,m}=\{[a]:a\in I_{k,N,m}\}.\]
Let $\varphi(a)=[a]$ be the quotient map from $I_{k,N,m}$ to $\mathcal{G}_{k,N,m}$.
\end{definition}

For each internal set $A\subseteq [k,(\lceil\!\sqrt[m]{k}\rceil+N)^m]$, we set 
$$\nu(A):=\st\left(\frac{1}{mN}\sum_{a\in A}\frac{1}{a^{\frac{m-1}{m}}}\right).$$  As before, we can extend $\nu$ to the $\sigma$-algebra generated by the internal sets.

\begin{proposition}
Let $A\subseteq\n$ and $\alpha>0$. Then $\BD_m(A)\geq\alpha$ if and only if there exists an
$I_{k,N,m}$ such that $\nu(^*\!A\cap I_{k,N,m})\geq\alpha$.
\end{proposition}

\begin{proposition}
Let $[a,b]\subseteq I_{k,N,m}$. Then
\[\nu([a,b])=\st\left(\frac{\sqrt[m]{b}-\sqrt[m]{a}}{H}\right).\]
Furthermore, if $c\in\starn$ is such that $(\lceil\sqrt[m]{b}\rceil+c)^m\in I_{k,N,m}$, then
\[\nu([(\lceil\sqrt[m]{a}\rceil+c)^m,(\lceil\sqrt[m]{b}\rceil+c)^m])
=\nu([a,b]).\]
\end{proposition}


\begin{definition}
For each set $E\subseteq\mathcal{G}_{k,N,m}$, we say that $E$ is $\mathfrak{m}$-measurable
if $\varphi^{-1}(E)$ is Loeb measurable, in which case we define the measure
\[\mathfrak{m}(E)=\nu(\varphi^{-1}(E)).\]
\end{definition}

\begin{theorem}
Let $U_N$ denote the largest additive cut in $[1,N]$ and fix $U<c<U_N$. For each $[a]\in\mathcal{G}_{k,N,m}$ set
\[T_c([a]):=[(\lceil\!\sqrt[m]{a}\rceil+c)^m].\] Then
$T_c$ is an $\mathfrak{m}$-measure preserving transformation on $\mathcal{G}_{k,N,m}$.
\end{theorem}

Note that if $\mathfrak{m}(E)>0$, then $E$ contains arbitrarily long sequences of the form
$[a], [(\lceil\!\sqrt[m]{a}\rceil+d\d)^m], [(\lceil\!\sqrt[m]{a}\rceil+2d)^m],
\ldots,[(\lceil\!\sqrt[m]{a}\rceil+ld)^m]$, i.e., $E$ contains arbitrarily long
$m$-th powers of arithmetic progressions.  Thus, using the techniques of the previous section, if $A\subseteq \n$ satisfies $\BD_m(A)>0$, then in $A$ we can find approximations to arbitrarily long sequences of $m$-th powers of arithmetic progressions.
\section{Lebesgue Density Theorem}

In this section, we fix $N>\n$ and a multiplicative cut $V$ contained in $[1,N]$.  Suppose that $A\subseteq [k,Nk]$ is internal and set $X:=\varphi_V(A).$  For $x\in\mathcal{H}_{k,N,V}$ and $r>V$, we write $\frak{m}_{x,r,V}(X)$ to denote $\frak{m}_{b,r,V}(X\cap [x,\varphi(r)x])$ for any $b\in \varphi_V^{-1}(\{x\})$; since $V\subseteq V_N$, we see, by the discussion preceding Proposition \ref{measurepreserving}, that the definition of $\mathfrak{m}_{x,r,V}$ is independent of the choice of representative of $\varphi_V^{-1}(\{x\})$.  We then set
\[
\delta_+(x,X)=\liminf_{r>V}\mathfrak{m}_{x,r,V}(X),
\]
or, equivalently, to clarify the meaning of $\lim\inf$ in this setting:
\[
\delta_+(x,X)=\sup_{s>V}\inf_{V<r<s}\mathfrak{m}_{x,r,V}(X).
\]

One can define the notion of $\delta_-(x,X)$ in an analogous fashion.  We say that $x\in \mathcal{H}_{k,N,V}$ is a \emph{Lebesgue density point of $X$} if $\delta_+(x,X)=\delta_-(x,X)=1$.

Here is the version of the Lebesgue Density Theorem in our setting.  We model our proof after a proof of the classical Lebesgue density theorem given by Faure in \cite{faure}.

\begin{theorem}\label{ldt}
\ Let $A$ be an internal subset of $[k,Nk]$ and $X=\varphi_V(A).$ Then
$\mathfrak{m}_{k,N,V}$-almost every point in $X$ is a Lebesgue density point.
\end{theorem}

\begin{proof}
We only show that almost every point $x$ of $X$ satisfies $\delta_+(x,X)=1$.  Fix $n$ and set $X_n:=\{x\in X \ : \ \delta_+(x,X)<\frac{n}{n+1}\}$.  It suffices to show that $\frak{m}_{k,N,V}^*(X_n)=0$.  (Here, $\frak{m}_{k,N,V}^*$ denotes the outer measure.)  Fix $\epsilon>0$.  Take internal sets $C\subseteq D\subseteq [k,Nk]$ such that $C\subseteq \varphi_V^{-1}(X)\subseteq D\subseteq {}^*\n\setminus \n$ and $\nu(D\setminus C)<\epsilon$.  (In this proof, we write $\nu$ for $\nu_{k,N}$.)  Fix $D'\subseteq D$ internal such that $\varphi_V^{-1}(X_n)\subseteq D'$ and such that $\nu(D')<\nu^*(\varphi_V^{-1}(X_n))+\epsilon$.  We now set
$$C':=\{a\in C \ : \ (\exists b\geq 2)([a,ba]\subseteq D' \text{ and } \frac{1}{\ln b}\sum_{x\in C\cap [a,ba]} \frac{1}{x}<\frac{n+1}{n+2})\}.$$  Note that $C'$ is internal and $C'\subseteq C\cap D'$.

We first claim that $\varphi_V^{-1}(X_n)\cap C\subseteq C'$.  Fix $a\in \varphi_V^{-1}(X_n)\cap C$.  Since $[a]^V\subseteq \varphi_V^{-1}(X_n)\subseteq D'$, there is $c>V$ such that $[a,ca]\subseteq D'$.  Since $\delta_+(\varphi_V(a),X)<\frac{n}{n+1}$, there is $V<b<c$ such that $\nu(\varphi_V^{-1}(X))<\frac{n+1}{n+2}$.  It follows that
$$\frac{1}{\ln b}\sum_{x\in C\cap [a,ba]}\frac{1}{x}\approx \nu_{a,b}(C)\leq \nu_{a,b}(\varphi_V^{-1}(X))<\frac{n+1}{n+2},$$  whence we conclude that $a\in C'$.

Since $\varphi_V^{-1}(X_n)\subseteq C' \cup (D\setminus C)$, we get $\nu^*(\varphi_V^{-1}(X_n))\leq \nu(C')+\epsilon$, so $\nu(D')-\nu(C')\leq \nu(D')-\nu^*(\varphi^{-1}(X_n))+\epsilon<2\epsilon$.

Without loss of generality, we may suppose that $D'$ has big components.  Write $D':=\bigsqcup_i [a_i,b_i]$ into its components.   We now claim that $\frac{1}{\ln(b_i)-\ln(a_i)}\sum_{x\in C'\cap [a_i,b_i]}\frac{1}{x}\leq \frac{n+1}{n+2}$ for each $i$.  Fix $i$ and let $e_i\in [a_i+2,b_i+1]$ be maximal such that $\frac{1}{\ln(e_i-1)-\ln(a_i)}\sum_{x\in C'\cap [a_i,e_i-1]}\frac{1}{x}\leq \frac{n+1}{n+2}$.  We want to show that $e_i=b_i+1$.  Suppose, towards a contradiction, that $e_i\leq b_i$.  First suppose that $e_i\in C'$.  Take $b\geq 2$ such that $[e_i,be_i]\subseteq D'$ and $\frac{1}{\ln b}\sum_{x\in C\cap [e_i,be_i]}\frac{1}{x}\leq \frac{n+1}{n+2}$.  Then
\begin{alignat}{2}
\sum_{x\in C'\cap [a_i,be_i]}\frac{1}{x}&=\sum_{x\in C'\cap [a_i,e_i-1]}\frac{1}{x}+\sum_{x\in C'\cap [e_i,be_i]}\frac{1}{x}\notag \\ \notag
						            &\leq \frac{n+1}{n+2}((\ln(e_i-1)-\ln(a_i))+\ln b)\notag \\
						            &\leq \frac{n+1}{n+2}(\ln(be_i)-\ln(a_i)).\notag
\end{alignat}
Since $[e_i,be_i]\subseteq D'$, we have $be_i\leq b_i$, so $be_i+1\leq b_i+1$ contradicts the maximality of $e_i$.  We now suppose that $e_i\notin C'$.  Then
\begin{alignat}{2}
\sum_{x\in C'\cap [a_i,e_i]}\frac{1}{x}&=\sum_{x\in C'\cap [a_i,e_i-1]} \frac{1}{x} \notag \\ \notag
							&\leq \frac{n+1}{n+2}(\ln(e_i-1)-\ln(a_i))\notag \\
							&\leq \frac{n+1}{n+2}(\ln (e_i)-\ln(a_i)).\notag
\end{alignat}

\noindent Thus $e_i+1$ also works, contradicting the choice of $e_i$.

We now can calculate:
\begin{alignat}{2}
\nu(D')&\leq \nu(C')+2\epsilon \notag \\ \notag
               &\approx \frac{1}{\ln N} \sum_{x\in C'\cap [k,Nk]}\frac{1}{x}+2\epsilon \notag \\
               &= \frac{1}{\ln N}\sum_i \sum_{x\in C'\cap [a_i,b_i]}\frac{1}{x}+2\epsilon \notag \\
               &\leq \frac{1}{\ln N}\sum_i \frac{n+1}{n+2}(\ln(b_i)-\ln(a_i))+2\epsilon \notag \\
               &\approx \frac{n+1}{n+2}\cdot  \nu(D')+2\epsilon. \notag
\end{alignat}

\noindent The last step used that $D'$ has big components and is contained in ${}^*\n\setminus \n$.

We now conclude that $\nu^*(\varphi^{-1}(X_n))\leq\nu(D')\leq 2(n+2)\epsilon$.  Since $\epsilon$ was arbitrary (but $n$ is fixed), we get that $\nu^*(\varphi^{-1}(X_n))=0$, so $\frak{m}_{k,N,V}^*(X_n)=0$, as desired.
\end{proof}

\section{Productset phenomenon}

In this section, we use the Lebesgue Density Theorem for multiplicative cuts to obtain a multiplicative analog of Jin's sumset result from \cite{jin}.  First, we establish some notation.  For $u\in [1,N]$, set $u^{-1}:=\lfloor \frac{N}{u}\rfloor$.  Of course, this notion depends on $N$ and occasionally we will want to make this dependence explicit, in which case we write $u^{-1,N}$.  

The first goal of this section is to prove the following:

\begin{theorem}\label{involution}
There is a map $\Upsilon=\Upsilon_{N,V}:\mathcal{H}_{1,N,V}\to \mathcal{H}_{1,N,V}$ given by $\Upsilon(\varphi_V(u)):=\varphi_V(u^{-1})$.  Moreover, $\Upsilon$ is an invertible measure-preserving transformation satisfying $\Upsilon^{-1}=\Upsilon$.
\end{theorem}

We break the proof of Theorem \ref{involution} up into a series of lemmas.  We first prove that $\Upsilon$ is well-defined.

\begin{lemma}\label{defined}
Suppose that $u,v\in [1,\frac{N}{2}]$ satisfy $u\sim_V v$.  Then $u^{-1}\sim_V v^{-1}$.
\end{lemma}

\begin{proof}
Without loss of generality, $u\leq v$.  We must show that $\lfloor \frac{u^{-1}}{v^{-1}}  \rfloor \in V$.  Write $u^{-1}:=\frac{N}{u}-\epsilon$ and $v^{-1}:=\frac{N}{v}-\delta$, where $\epsilon,\delta \in [0,1)$.  Then:
$$\frac{u^{-1}}{v^{-1}}=\frac{N-\epsilon u}{N-\delta v}\cdot \frac{v}{u}\leq \frac{N}{N-v}\cdot \frac{v}{u}\leq 2\cdot \frac{v}{u}.$$
\end{proof}

We next prove that $\Upsilon$ is an involution.

\begin{lemma}\label{invo}
Suppose that $x\in [1,\frac{N}{2}]$.  Then $x\sim_V (x^{-1})^{-1}$.
\end{lemma}

\begin{proof}
Since $x^{-1}\leq \frac{N}{x}$, we have $x\leq \frac{N}{x^{-1}}$, so $x\leq (x^{-1})^{-1}$.  Write $(x^{-1})^{-1}=\frac{N}{x^{-1}}-\delta_1$ and $x^{-1}=\frac{N}{x}-\delta_2$, with $\delta_1,\delta_2\in [0,1)$.  We then have:
$$\frac{(x^{-1})^{-1}}{x}=\frac{\frac{N}{\frac{N}{x}-\delta_2}-\delta_1}{x}=\frac{Nx-\delta_1N+\delta_1\delta_2x}{x(N-\delta_2x)}\leq \frac{N}{N-x}\leq 2.$$
\end{proof}

Suppose that $A\subseteq [1,N]$ is internal and its decomposition into components is $A=\bigsqcup_{i\in I}[a_i,b_i]$.  We say that $A$ has \emph{separated components} if, whenever $[a_i,b_i]$ and $[a_j,b_j]$ are adjacent components with $a_j>b_i$, we have $a_j>2b_i$. 

\begin{lemma}
Suppose that $A$ has separated components and is contained in $\bigcap_{k\in \n}[1,\frac{N}{k})$.  Then, for any distinct $i,j\in I$, we have $[b_i^{-1},a_i^{-1}]\cap [b_j^{-1},a_j^{-1}]=\emptyset$. 
\end{lemma}

\begin{proof}
Without loss of generality, assume that $2b_i<a_j$.  Suppose that $b_i^{-1}\leq x \leq a_i^{-1}$.  Then $\frac{N}{b_i}-\epsilon\leq x\leq \frac{N}{a_i}$ for some $\epsilon\in [0,1)$.  We then have $a_i\leq \frac{N}{x}\leq \frac{Nb_i}{N-b_i}$, so $a_i\leq x^{-1}\leq 2b_i$ since $\frac{b_i}{N}\approx 0$.  If $b_j^{-1}\leq x\leq a_j^{-1}$, then we would have $a_j\leq x^{-1}$, contradicting $2b_i<a_j$.
\end{proof}

For internal $A\subseteq [1,N]$ with decomposition $A=\bigsqcup_{i\in I} [a_i,b_i]$, we set $A^{-1}=\bigsqcup_{i\in I} [b_i^{-1},a_i^{-1}]$.  If $A$ has separated components and is contained in $\bigcap_{k\in \n}[1,\frac{N}{k})$, the preceding lemma tells us this definition of $A^{-1}$ is also its decomposition into components.

\begin{lemma}\label{invappr}
Suppose that $A\subseteq [1,N]$ is internal, has big and separated components, and is contained in $({}^*\n\setminus \n)\cap \bigcap_{k\in \n}[1,\frac{N}{k})$.  Then $A^{-1}$ has big components and $\nu(A)= \nu(A^{-1})$.
\end{lemma}

\begin{proof}
In order to show that $A^{-1}$ has big components, it suffices to show that if $[a,b]$ is big and $\frac{b}{N}$ is infinitesimal, then $[b^{-1},a^{-1}]$ is also big.  Write $a^{-1}=\frac{N}{a}-\epsilon$ and $b^{-1}=\frac{N}{b}-\delta$.  Then:
$$\frac{a^{-1}}{b^{-1}}=\frac{b}{a}\cdot \frac{N-\epsilon a}{N-\delta b}>\frac{b}{a}\cdot (1-\frac{a}{N}).$$  The quantity on the right hand side of the display is appreciably larger than $2$ since $\frac{b}{a}$ is appreciably larger than $2$ and $\frac{a}{N}$ is infinitesimal.  

We now must show that $\nu(A)= \nu(A^{-1})$.  Decompose $A=\bigsqcup_{i\in I}[a_i,b_i]$ into its components; then $[b_i^{-1},a_i^{-1}]$ are the components of $A^{-1}$.  By Lemma \ref{bigestimate} (which applies to $A^{-1}$ since $A\subseteq \bigcap_{k\in \n}[1,\frac{N}{k})$), we know that
$$\nu(A)\approx \frac{1}{\ln N}\sum_{i\in I}(\ln(b_i)-\ln(a_i))$$ and 
$$\nu(A^{-1})\approx \frac{1}{\ln N}\sum_{i\in I}(\ln(a_i^{-1})-\ln(b_i^{-1})).$$
For simplicity, set $\alpha_i:=\ln(b_i)-\ln(a_i)$ and $\beta_i:=\ln(a_i^{-1})-\ln(b_i^{-1})$.  Fix $i\in I$ and write $a_i^{-1}=\frac{N}{a_i}-\epsilon$ and $b_i^{-1}=\frac{N}{b_i}-\delta$.  Then $|\alpha_i-\beta_i|=|\ln(\frac{N-\epsilon a_i}{N-\delta b_i})|\approx 0$.  Since $A$ has big components, it follows that $\frac{|\alpha_i-\beta_i|}{\alpha_i}\approx 0$.  It follows that
$$|\frac{\sum_{i\in I}\alpha_i}{\ln N}-\frac{\sum_{i\in I}\beta_i}{\ln N}|\leq \frac{\sum_{i\in I}|\alpha_i-\beta_i|}{\ln N}\leq \frac{\sum_{i\in I}|\alpha_i-\beta_i|}{\sum_{i\in I}\alpha_i}\approx 0.$$  Putting everything together, we get $\nu(A)= \nu(A^{-1})$.
\end{proof}


\begin{lemma}\label{invpres}
Suppose that $E\subseteq \mathcal{H}_{1,N,V}$ is $\frak{m}_{1,N,V}$-measurable.  Then $\Upsilon(E)$ is $\frak{m}_{1,N,V}$-measurable and $\mathfrak{m}_{1,N,V}(\Upsilon(E))=\mathfrak{m}_{1,N,V}(E)$
\end{lemma}


\begin{proof}
Without loss of generality, $\varphi_V^{-1}(E)\subseteq ({}^*\n\setminus \n)\cap \bigcap_{k\in \n}[1,\frac{N}{k})$.  Fix $\epsilon>0$ and take internal sets $C\subseteq \varphi_V^{-1}(E)\subseteq D$ with $\nu_{1,N}(D\setminus C)<\epsilon$.  Without loss of generality, $C,D\subseteq ({}^*\n\setminus \n)\cap \bigcap_{k\in \n}[1,\frac{N}{k})$ and both $C$ and $D$ have big and separated components.  Decompose $C=\bigsqcup_{i\in I} [a_i,b_i]$ and $D=\bigsqcup_{j\in J} [c_j,d_j]$ into their measurable components.

\

\noindent \textbf{Claim:}  $C^{-1}\subseteq \varphi^{-1}_V(\Upsilon(E))\subseteq D^{-1}$.

\noindent \textbf{Proof of Claim:}  First suppose that $x\in [b_i^{-1},a_i^{-1}]$.  Write $b^{-1}=\frac{N}{b}-\delta$ for some $\delta \in[0,1)$.  Then $$a\leq \frac{N}{x}\leq \frac{Nb}{N-\delta b}\leq \frac{Nb}{N-b}\leq 2b.$$  Since $b\sim_V2b$, we have $\varphi_V(x^{-1})\in E$.  Since $a_i\geq 2$, we have $x\leq a_i^{-1}\leq \frac{N}{2}$, so $x\sim_V(x^{-1})^{-1}\in \varphi_V^{-1}(\Upsilon(E))$ and thus $x\in \varphi_V^{-1}(\Upsilon(E))$.  Now suppose that $x\in \varphi_V^{-1}(\Upsilon(E))$.  Then $x\sim_Vu^{-1}$ for some $u\in \varphi_V^{-1}(E)$.  Choose $j\in J$ such that $u\in [c_j,d_j]$.  Since $d_j+1\notin D$, we cannot have $u\sim_Vd_j$.  Now since $u,x^{-1}\in [1,\frac{N}{2}]$, we have $u\sim_V(u^{-1})^{-1}\sim_V x^{-1}$, whence $x^{-1}\leq d_j$.  Note that $x^{-1}<d_j$, else we contradict $d_j+1\notin D$.  It follows that $\frac{N}{x}\leq d_j$, so $\frac{N}{d_j}\leq x$, whence $d_j^{-1}\leq x$.  Similarly, $u\not\sim_Vc_j$, so $c_j\leq x^{-1}\leq \frac{N}{x}$.  It follows that $x\leq \frac{N}{c_j}$, so $x\leq c_j^{-1}$.  This completes the proof of the claim.

\

%
%

By Lemma \ref{invappr}, we have that $\nu(C^{-1})= \nu(C)$ and $\nu(D^{-1})= \nu(D)$.  Once again, it follows that $\varphi_V^{-1}(\Upsilon(E))$ is measurable and has the same measure as $E$.
\end{proof}

Note that Lemmas \ref{defined}, \ref{invo}, and \ref{invpres} together establish Theorem \ref{involution}.

\begin{lemma}\label{nsjin}
Suppose that $A$ is an internal subset of $[j,Nj]$ and $B$ is an internal subset of $[k,Nk]$.  Set $X=\varphi_{V}(A)$ and $Y=\varphi_{V}(B)$.  Suppose that $\mathfrak{m}_{j,N,V}(X)>0$ and $\mathfrak{m}_{k,N,V}(Y)>0$.  Then $XY$ contains
a non-empty interval in $\mathcal{H}_{jk,N^2,V}$.
\end{lemma}

\begin{proof}
Let $x\in X$ and $y\in Y$ be Lebesgue density points of $X$ and $Y$
respectively. \ Then there exists $r>V$ such that
\[
\mathfrak{m}_{x,r,V}(X\cap\lbrack x,xr])>\frac{2}{3}
\]
and
\[
\mathfrak{m}_{\frac{y}{r},r,V}(Y\cap\lbrack \frac{y}{r},y])>\frac{2}{3}.
\]

\noindent Here, and in the rest of this proof, $\frac{y}{r}$ denotes $\varphi_V(\lfloor \frac{a}{r}\rfloor)$ for any $a\in \varphi_V^{-1}(\{y\})$.  We now set $$E_{X}:=\left\{  u\in\varphi_V\left(  \lbrack1,r]\right)  :ux\in
X\right\}  $$ and $$E_{Y}:=\left\{  v\in\varphi_V\left(  \lbrack1,r]\right)
:\frac{y}{v}\in Y\right\}  .$$  Note that
\[
T_{x}(E_{X})=X\cap\lbrack x,xr]\text{, and }T_{\frac{y}{r}}(\Upsilon_r(E_Y)
)=Y\cap\lbrack r^{-1}y,y].
\]
By Proposition \ref{measpresgen} and Lemma \ref{invpres}, we have that $\mathfrak{m}%
_{1,r,V}(E_{X})>2/3$ and $\mathfrak{m}_{1,r,V}(E_{Y})>2/3$.

In order to finish the proof of the theorem, we show that $xys\in XY$ for any $s$ satisfying $V<s<r^{1/3}$.  Towards this end, consider
the set $E_{X}^{\prime}:=\left\{  u\in\varphi_V\left(  \lbrack1,r]\right)  :usx\in
X\right\}$.
Then
\[
E_{X}\cap\lbrack s,r]\subset T_{s}\left(  E_{X}^{\prime}\cap\lbrack
1,\frac{r}{s}]\right)
\]
so that
\begin{alignat}{2}
\mathfrak{m}_{1,r,V}(E_{X}^{\prime})&\geq\mathfrak{m}_{1,r,V}(T_{s}\left(E_{X}^{\prime}\cap\lbrack1, \frac{r}{s}]\right)\notag \\ \notag
						       &\geq\mathfrak{m}_{1,r,V}(E_{X}\cap\lbrack s,r])\notag \\
						       &>2/3-\mathfrak{m}_{1,r,V}([1,s])\notag \\
						       &>1/3.\notag
\end{alignat}
Since $\mathfrak{m}_{1,r,V}(E_{X}^{\prime})+\mathfrak{m}_{1,r,V}(E_{Y})>1$,
there exists $u_{0}\in E_X^{\prime}\cap E_Y$. \ Then $u_{0}sx$ is in
$X$ and $\frac{y}{u_0}$ is in $Y.$ \ Thus $sxy\in XY$, 
as desired.
\end{proof}

\noindent We now obtain a multiplicative analog of the main result of \cite{jin}:

\begin{theorem}
Suppose that $A,B\subseteq\mathbb{N}$ satisfy
$\lBD(A),\lBD(B)>0.$ \ Then there exists $m\in\mathbb{N}$ such that for all
$n\in\mathbb{N}$, there is $x\in \n$ such that,
for every $[u,mu]\subseteq\lbrack x,nx]$, we have $\lbrack u,mu]\cap (A\cdot B)\not=\emptyset$.
\end{theorem}

\begin{proof}
We work with the cut $V=\mathbb{N}$. \ Fix $N>\n$; by Proposition \ref{nsubld}, there exists
$j,k\in {}^*\n$ such that $\nu_{j,N}(^{\ast}A\cap [j,jN])>0$ and
$\nu_{k,N}(^{\ast}B\cap [k,kN])>0$. \ Let $X:={}^{\ast }\!{A}\cap [j,jN]\subseteq \mathcal{H}_{j,N}$ and $Y:={}^*B\cap [k,kN]\subseteq \mathcal{H}_{k,N}$.  By Lemma \ref{nsjin}, $XY$ contains a nonempty interval
in $\mathcal{H}_{jk,N^2}$, say $\varphi([a,b])$ with $\frac{b}{a}>\n$.

Let $\left\{  c_{i} \ : \ i\leq M\right\}  $ enumerate $^{\ast}(A\cdot B)\cap\lbrack a,b]$ in increasing order and let $m:=\max_{i<M}\left\{  \left\lceil \frac{c_{i+1}}{c_{i}%
}\right\rceil \right\}  $.  Then $m\in \n$, else $X$ would not contain the entire interval $\varphi([a,b])$.  We claim that this $m$ is as desired.  Indeed, given any $n\in\mathbb{N}$,  we have $b\geq na$ and for any interval $[u,mu]\subseteq [a,na]$ we have $[u,mu]\cap {}^*(A\cdot B)\not=\emptyset$, whence we obtain the existence of the desired $x\in \n$ by transfer.
\end{proof}

\providecommand{\bysame}{\leavevmode\hbox to3em{\hrulefill}\thinspace}
\providecommand{\MR}{\relax\ifhmode\unskip\space\fi MR }
\providecommand{\MRhref}[2]{%
  \href{http://www.ams.org/mathscinet-getitem?mr=#1}{#2}
}
\providecommand{\href}[2]{#2}


\begin{thebibliography}{10}

\bibitem{beiglock}
Mathias Beiglb\"{o}ck, \emph{An ultrafilter approach to {J}in's theorem}, Israel
  Journal of Mathematics \textbf{185} (2011), no.~1, 369--374.

\bibitem{BBF}
Mathias Beiglb\"{o}ck, Vitaly Bergelson, and Alexander Fish, \emph{Sumset
  phenomenon in countable amenable groups}, Advances in Mathematics
  \textbf{223} (2010), no.~2, 416--432.

\bibitem{BBHS2}
Mathias Beiglb\"{o}ck, Vitaly Bergelson, Neil Hindman, and Dona Strauss,
  \emph{Multiplicative structures in additively large sets}, Journal of
  Combinatorial Theory, Series A \textbf{113} (2006), no.~7, 1219--1242.

\bibitem{bergelson2}
Vitaly~Bergelson, \emph{Sets of recurrence of {$\mathbb{Z}^m$}-actions and
  properties of sets of differences in {$\mathbb{Z}^m$}}, Journal of the London
  Mathematical Society \textbf{s2-31} (1985), no.~2, 295--304.

\bibitem{bergelson1}
Vitaly Bergelson, \emph{Multiplicatively large sets and ergodic {R}amsey
  theory}, Israel Journal of Mathematics \textbf{148} (2005), no.~1, 23--40.

\bibitem{tao}
Emmanuel Breuillard, Ben Green, and Terence Tao, \emph{Approximate subgroups of
  linear groups}, Geometric and Functional Analysis \textbf{21} (2011), no.~4,
  774--819.

\bibitem{dinasso}
Mauro Di~Nasso, \emph{An elementary proof of {J}in's theorem with a bound},
  Electronic Journal of Combinatorics \textbf{21} (2014), no.~2, Paper 2.37, 7.

\bibitem{DGJLLM}
Mauro Di~Nasso, Isaac Goldbring, Renling Jin, Steven Leth, Martino Lupini, and
  Karl Mahlburg, \emph{High density piecewise syndeticity of sumsets},
  {arXiv}:1310.5729 (2013).

\bibitem{faure}
Claude-Alain Faure, \emph{A short proof of {L}ebesgue's density theorem}, The
  American Mathematical Monthly \textbf{109} (2002), no.~2, 194--196.

\bibitem{Fekete}
Michael Fekete, \emph{{{\"{U}}ber die Verteilung der Wurzeln bei gewissen
  algebraischen Gleichungen mit ganzzahligen Koeffizienten}}, Mathematische
  Zeitschrift \textbf{17} (1923), no.~1, 228--249.

\bibitem{udi}
Ehud Hrushovski, \emph{Stable group theory and approximate subgroups}, Journal
  of the American Mathematical Society \textbf{25} (2012), no.~1, 189--243.

\bibitem{jin}
Renling Jin, \emph{The sumset phenomenon}, Proceedings of the American
  Mathematical Society \textbf{130} (2002), no.~3, 855--861.

\bibitem{jinintro}
\bysame, \emph{Introduction of nonstandard methods for number theorists},
  Integers. Electronic Journal of Combinatorial Number Theory \textbf{8}
  (2008), no.~2, A7, 30.

\bibitem{JK}
Renling Jin and H.~Jerome Keisler, \emph{Abelian groups with layered tiles and
  the sumset phenomenon}, Transactions of the American Mathematical Society
  \textbf{355} (2003), no.~1, 79--97.

\end{thebibliography}
\end{document}